\documentclass{amsart}
\usepackage{amssymb, amsmath, amsthm, graphics, comment, xspace, enumerate}
\newtheorem{thm}{Theorem}[section]

\newtheorem{prop}[thm]{Proposition}
\theoremstyle{definition}
\newtheorem{defn}[thm]{Definition}

\theoremstyle{remark}
\newtheorem{rem}[thm]{Remark}
\numberwithin{equation}{section}

\newcommand{\NN}{\mathbb N}

\newcommand{\RR}{\mathbb R}

\newcommand{\EE}{\mathcal E}
\newcommand{\DD}{\mathcal D}

\begin{document}


\vskip-2.5cm
\title{On the exponential ultradistribution semigroups in Banach spaces}

\author{Marko Kosti\' c}
\address{Faculty of Technical Sciences,
University of Novi Sad,
Trg D. Obradovi\' ca 6, 21125 Novi Sad, Serbia}
\email{marco.s@verat.net}

\author{Stevan Pilipovi\' c}
\address{Department for Mathematics and Informatics,
University of Novi Sad,
Trg D. Obradovi\' ca 4, 21000 Novi Sad, Serbia}
\email{pilipovic@dmi.uns.ac.rs}

\author{Daniel Velinov}
\address{Department for Mathematics, Faculty of Civil Engineering, Ss. Cyril and Methodius University, Skopje,
Partizanski Odredi
24, P.O. box 560, 1000 Skopje, Macedonia}
\email{velinovd@gf.ukim.edu.mk}

\newcommand{\FilMSC}{Primary 47D06; Secondary 47D60, 47D99.}
\newcommand{\FilKeywords}{Exponential ultradistribution semigroups; Holomorphic exponential ultradistribution semigroups; Tempered ultradistributions; Quasi-equicontinuous semigroups.}
\newcommand{\FilCommunicated}{Dragan S. Djordjevi\'c}
\newcommand{\FilSupport}{The first two authors are supported by Grant No. 177024 of Ministry of Science and Technological Development, Republic of Serbia. The third author is supported by ...}

\begin{abstract}
In this paper, we continue our previous research studies of exponential ultradistribution semigroups in Banach spaces. The existence and uniqueness of analytical solutions of abstract fractional relaxation equations associated with the generators of exponential ultradistribution semigroups have been considered. Some other results are also proved.
\end{abstract}

\maketitle

\section{Introduction and preliminaries }
The class of (exponential) distribution semigroups in Banach spaces was introduced by J. L. Lions in \cite{li121}. Along with the papers \cite{a11} by W. Arendt on integrated semigroups and vector-valued Widder representation theorem, and \cite{c62} by G. Lumer and I. Cioranescu on $K$-convoluted semigroups, the paper \cite{li121} has been a motivation for
a great number of other authors to study various classes of (ultra-)distribution semigroups in the setting of Banach spaces or, more generally, locally convex spaces (cf. \cite{a43}, \cite{a22}, \cite{l115}-\cite{keya}, \cite{kniga}, \cite{novi}, \cite{me152}-\cite{n181}, \cite{yoshinaga}-\cite{w241} and references cited therein for further information). Following the approaches employed in \cite{ku112}-\cite{li121}, the authors of \cite{kps} have defined the classses of L-ultradistribution semigroups and ultradistribution semigroups with non-densely defined generators. In \cite{kpsv}, some structural characterizations of ultradistribution semigroups and exponential ultradistribution semigroups have been proved. Motivated by the ideas from D. Fujiwara's research of exponential distribution semigroups \cite{fuji}, we shall prove some new characterizations of spaces of vector-valued tempered ultradistributions in Proposition \ref{lora}. In Theorem \ref{glass}, we consider the existence and uniqueness of analytical solutions of abstract fractional relaxation equations associated with the generators of exponential ultradistribution semigroups. Some new qualitative properties of solutions of the first order abstract Cauchy problem $
(ACP)$ clarified below have been proved in the case that the operator $A$ generates an exponential ultradistribution semigroup of $(M_{p})$-class and that the initial value
$x$ belongs to the abstract Beurling space associated with the operator $A.$

We use the standard terminology throughout the paper. We impose the following conditions (cf. \cite{k82}) on the sequence $(M_p)_{p\in {\mathbb N}_{0}}$ of positive real numbers satisfying $M_{0}=1$:\\
$(M.1)$ $M^2_{p}\leq M_{p-1}M_{p+1}$ for $p\in\NN$;\\
$(M.2)$ For some $A,H>0$,
$M_p\leq AH^p\min_{0\leq q\leq p}M_{p-q}M_q,\quad p,q\in\NN;$\\
$(M.3)':$
$$
\sum_{p=1}^{\infty}\frac{M_{p-1}}{M_p}<\infty.
$$
Every employment of the condition\vspace{0.1cm} \newline \noindent
$(M.3):$
$$
\sup_{p\in\mathbb{N}}\sum_{q=p+1}^{\infty}\frac{M_{q-1}M_{p+1}}{pM_pM_q}<\infty,
$$
which is a slightly stronger than $(M.3)',$ will be explicitly emphasized.
The Gevrey sequence $M_p=p!^s$, $s>1$ satisfies all the above conditions. The associated function $M(\rho)$ of sequence $(M_p)$ on $(0,+\infty)$ is defined by
$M(\rho)=\sup_{p\in\NN}\ln \frac{\rho^p}{M_p},\ \rho>0.$ We know that he
function $t\mapsto M(t)$, $t\geq 0$ is increasing as well as that
$\lim_{\lambda\to\infty}M(\lambda)=\infty$ and that the
function $M(\cdot)$ vanishes in some open neighborhood of zero. Put $m_{p}:=M_{p}/M_{p-1},$ $p\in {\mathbb N}.$ Then $(M.1)$ implies that the sequence $(m_{p})_{p\in {\mathbb N}}$ is increasing.\\
\indent The spaces of ultradifferentiable functions of $\ast$-class and ultradistributions of $\ast$-class are defined in the sequel. Let $\emptyset \neq \Omega \subseteq \RR,$ and let $\emptyset \neq K\Subset\Omega$ denote a compact set in $\Omega$.
Set, for every $h>0,$
$$
\DD_K^{\{M_p\},h}:=\bigl\{\varphi\in C^{\infty}(\RR^n)  : \mbox{supp}(\varphi) \subseteq K\mbox{ and there exists }C>0\mbox{ such that }\|D^{\alpha}\varphi\|_{C(K)}\leq Ch^{-\alpha}M_{\alpha},\ \alpha=0,1,2,\cdot \cdot \cdot \bigr \}.
$$
Equipped with the norm $\|\cdot\|_{\DD_K^{\{M_p\},h}}:=\sup_{x\in K,\alpha \in {\mathbb N}_{0}}\frac{|D^{\alpha}\cdot(x)|}{h^{-\alpha}M_{\alpha}},$
$\DD^{\{M_p\},h}_{K}$ becomes a Banach space. Define also, as locally convex spaces,
$\DD_K^{(M_p)}:=\lim{\rm  proj}_{h\rightarrow +\infty }\DD_K^{\{M_p\},h}$,
$\DD^{(M_p)}(\Omega):=\lim{\rm  ind}_{{K\Subset\Omega}}\DD_K^{(M_p)}$,
$\DD_K^{\{M_p\}}:=\lim{\rm  ind}_{h\rightarrow 0}\DD_K^{\{M_p\},h}$ and
$\DD^{\{M_p\}}(\Omega):=\lim{\rm  ind}_{{K\Subset\Omega}}\DD_K^{\{M_p\}}$. The notions of spaces
$\EE^{(M_p)}(K)$,
$\EE^{(M_p)}(\Omega)$,
$\EE^{\{M_p\}}(K)$ and
$\EE^{\{M_p\}}(\Omega)$ will be understood in the sense of \cite{k91}.

The spaces of tempered ultradistributions of the Beurling,
resp., the Roumieu type, are defined in \cite{pilip} as duals of the following test spaces
$$
\mathcal{S}^{(M_p)}(\mathbb{R}):=\lim \text{proj}_{h\to\infty}\mathcal{S}^{M_p,h}(\mathbb{R}),
\mbox{ resp., }\mathcal{S}^{\{M_p\}}(\mathbb{R}):=\lim \text{ind}_{h\to 0}\mathcal{S}^{M_p,h}(\mathbb{R}),
$$
where for each $h>0,$
\begin{align*}
\mathcal{S}^{M_p,h}(\mathbb{R}):=\bigl\{\phi\in C^\infty(\mathbb{R}):\|\phi\|_{M_p,h}<\infty\bigr\}
\end{align*}
and
\begin{align*}
\|\phi\|_{M_p,h}:=\sup\Biggl\{\frac{h^{\alpha+\beta}}{M_\alpha M_\beta}\bigl(1+t^2\bigr)^{\beta/2}\bigl|\phi^{(\alpha)}(t)\bigr|:t\in\mathbb{R},
\;\alpha,\;\beta\in\mathbb{N}_0\Biggr\}.
\end{align*}
The  common  notation for  symbols  $(M_{p})$  and
$\{M_{p}\} $ will be $\ast .$

It is said that a function $P(\xi ) = \sum _{\alpha \in {\mathbb N_{0}}} a_{\alpha } \xi ^{\alpha
} ,\ \xi \in {\RR}$
is an ultrapolynomial of the class $(M_{p}) ,$ resp. $\{M_{p}\}$,
if the coefficients $a_{\alpha }$ satisfy the estimate
\begin{equation}\label{ocaa}
\mid a_{\alpha } \mid  \leq \frac{C L^{\alpha }}{M_{\alpha}} ,\ \alpha \in {\NN_{0}} ,
\end{equation}
for  some  $L > 0 $ and $C>0,$ resp. for every $L > 0 $  and  some
$C_{L} > 0. $
The  corresponding  operator  $P(D) =  \sum_{\alpha   \in {\mathbb N_{0}}}  a_{\alpha
}D^{\alpha  }$  is said to be an ultradifferential operator  of  the  class
$(M_{p})$, resp. $\{M_{p}\}$ (see \cite{k91}).

Let $X$ and $Y$ be two Hausdorff sequentially complete locally convex spaces over the field of complex numbers. By $L(X,Y)$ we denote the space consisting of all continuous linear mappings from $X$ into $Y;$ $L(X)\equiv L(X,X).$
A $C_0$ semigroup $(T(t))_{t\geq 0}$ on $X$ (cf. \cite{komura} and \cite{yosida} for the notion) is called locally equicontinuous if the family $\{T(t)\, : \, t\in[0,t_0]\}$ is equicontinuous for all $t_0>0$; $(T(t))_{t\geq 0}$  is called (globally) equicontinuous if the family $\{T(t) \, : \, t\geq0\}$ is equicontinuous. Recall that any $C_{0}$-semigroup in a barreled space $X$ is automatically locally equicontinuous (\cite{komura}). If for some $\alpha\in{\mathbb R}$ the $C_{0}$-semigroup $(T_{\alpha}(t)\equiv e^{-\alpha t}T(t))_{t\geq 0}$ is equicontinuous, then $(T(t))_{t\geq 0}$ is called $\alpha$ quasi--equicontinuous.

Henceforth we shall always assume that $(E,\| \cdot\|)$ is a complex Banach space. Let $A$ be a closed linear operator with domain $D(A)$ and range $R(A)$ contained in $E.$ Since no confusion
seems likely, we will identify $A$ with its graph; $[D(A)]$ denotes the Banach space $D(A)$ equipped with the graph norm $\|x\|_{[D(A)]}:=\|x\|+\|Ax\|,$ $x\in D(A).$ As is usually the case, $\rho(A)$ denotes the resolvent set of $A$ and $R(\lambda : A)$ denotes the operator $(\lambda I-A)^{-1}$ ($\lambda \in \rho(A)$), where $I$ stands for the identity operator on $E.$ If $F$ is a linear subspace of $E$, then we define the part of $A$ in $F,$ $A_{|F}$ for short, by $A_{|F}:=\{(x,y) \in A : x,y\in F\}.$ Set $D_{\infty}(A):= \bigcap_{n=1}^{\infty}D(A^n)$ and $\Sigma_{\alpha}:=\{ z\in {\mathbb C} \setminus \{0\} :
|\arg (z)|<\alpha \}$ ($\alpha \in (0,\pi]$). We denote by ${\mathcal A}_{\alpha}(E)$ the vector space consisting of all analytic mappings from $\Sigma_{\alpha}$ into $E$  ($\alpha \in (0,\pi]$).
The following system of seminorms $\|\cdot\|_{n}\equiv \|\cdot\|+\cdot \cdot \cdot+\|A^n \cdot\|$, $n\in{\mathbb N}_{0}$ turns $D_{\infty}(A)$ into a Fr\' echet space. We define the abstract Beurling space of $(M_p)$ class
associated to a closed linear operator $A$ as in \cite{ci1}.
Put $E^{(M_p)}(A):=\lim \mbox{proj}_{h\to+\infty}E^{\{M_p\}}_h(A)$, where
$$
E^{\{M_p\}}_h(A):=\Biggl\{x\in D_{\infty}(A):\|x\|^{\{M_p\}}_h \equiv \sup_{p\in\mathbb{N}_0}\frac{h^p\|A^px\|}{M_p}<\infty\Biggr\}.
$$
Then $(E^{\{M_p\}}_h(A)$, $\|\cdot\|^{\{M_p\}}_h)$ is a
Banach space, $E^{\{M_p\}}_{h'}(A)\subseteq E^{\{M_p\}}_h(A)$
if $0<h<h'<\infty,$ $E^{(M_p)}(A)$ is a Fr\' echet space, and $E^{(M_p)}(A)$ is a dense subspace of $E$
whenever $A$ is the generator of a regular $(M_p)$-ultradistribution semigroup (\cite{ci1}).

By ${\mathcal D}'^{*}_+({\mathbb R},L(E))$ we denote the space consisting of those $L(E)$-valued ultradistributions $G$ of $\ast$-class for which supp$(G)\subseteq [0,\infty)$ (cf.
\cite{kniga} for the notion) and
by $\DD_0^{\ast}$ we denote the space consisting of those ultradifferentiable functions $\varphi$ of $\ast$-class
such that $\varphi(t)=0$ for $t\leq 0.$ We refer the reader to \cite{k91}-\cite{k82} for further information concerning the convolution of (vector-valued) ultradifferentiable functions and ultradistributions of $\ast$-class.

\begin{defn} (\cite{kps}) \label{d6.1}
Let $G\in {\mathcal D}'^{*}_+({\mathbb R},L(E)).$ Then it is said that $G$ is an
L-ultradistribution semigroup of $*$-class
if:
 \newline
$ (U.1)\;\;  G(\phi*\psi) = G(\phi)G(\psi),\quad \phi, \psi \in
{\mathcal D}^{*}_0;\; $
\newline
$ (U.2)\;\;    {\mathcal N} (G) := \bigcap_{\phi\in {\mathcal
D}^{*}_0} N (G(\phi)) = \{0\};\;$
 \newline
 $ (U.3)\;\;   {\mathcal
R}(G) :=
 \bigcup_{\phi \in {\mathcal D}^{*}_0}
R(G(\phi)) \mbox{ is dense in } E;$
\newline
$ (U.4)$  For every $x \in {\mathcal R}(G)$ there exists a function
$u \in C([0,\infty),E)$ satisfying
$$
u(0) = x \mbox{ and } G(\phi)x =
\int\limits_{0}^{\infty}\phi(t)u(t)\, dt,\ \phi \in {\mathcal
D}_0^{*}.
$$
If $G\in {\mathcal D}'^{*}_+({\mathbb R},L(E))$ satisfies
\newline
$ (U.5 )\;\;   G(\phi*_0\psi) = G (\phi)G(\psi),\quad \phi,\psi \in
{\mathcal  D}_0^{*}\;$, where
$$
(f\ast_{0}g)(t):=\int
\limits^{t}_{0}f(t-s)g(s)\, ds,\; t\in\mathbb{R},
$$
then it is said that $G$ is a pre-(UDSG) of
$*$-class. If $(U.5)$ and $(U.2)$ are fulfilled for $G$, then $G$ is called
an ultradistribution semigroup of $*$-class, in short, (UDSG). A
pre-(UDSG) $G$ is said to be dense if it additionally satisfies $(U.3)$.
\end{defn}

\begin{defn}\label{d62}
An L-ultradistribution semigroup $G$ of $*$-class is said to be exponential,
EL-ultradistribution semigroup of $*$-class in short, if in addition to
$(U.1)-(U.4),$ $G$ fulfills:
$$
(U.6)\;\;(\exists a\geq 0) (e^{-a\cdot }G\in {\mathcal
S}'^{*}({\mathbb R},L(E))).\;
$$
Then we say that $G$ is of order $a$.
Conditions $(U.6),$  $(U.5)$, resp., $(U.6),$ $(U.5)$ and $(U.2)$,
define an exponential pre-(UDSG), resp., exponential (UDSG), and
they are denoted by pre-(EUDSG), resp., (EUDSG).
\end{defn}

The definition of (infinitesimal) generator $A$ of an ultradistribution semigroup $G$ of $\ast$-class is defined by $A:=\{(x,y)\in E\times E : G(-\varphi^{\prime})x=G(\varphi)y\mbox{ for all }\varphi \in \DD_0^{\ast}\}.$ Since every L-ultradistribution semigroup $G$ of $\ast$-class is also an ultradistribution semigroup of $\ast$-class (\cite{kps}), the definition of (infinitesimal) generator of such a semigroup is clear. We refer the reader to \cite{kps} for the notion of an (exponential) ultradistribution fundamental solution for a closed linear operator $A$ acting on $E.$

For the sake of convenience, we shall remind the reader of the following important result.

\begin{thm} (\cite{ieop}) \label{holh} Suppose $A$ is a closed linear operator on $E$. Then there exists an exponential ultradistribution fundamental solution of $\ast$-class for $A$ iff there exist $a\geq0$, $k>0$ and $L>0$, in the Beurling case, resp., there exists $a\geq0$ such that, for every $k>0$ there exists $L_k>0$, in the Roumieu case, such that: $$\{\lambda\in{\mathbb C} : \Re\lambda>a\}\subseteq\rho(A), \mbox{ and}$$
$$\|R(\lambda:A)\|\leq Le^{M(k|\lambda|)},\lambda\in{\mathbb C},\ \Re\lambda>a, resp.,$$
$$\|R(\lambda:A)\|\leq L_k e^{M(k|\lambda|)}\quad\mbox{for all}\quad k>0 \quad\mbox{and}\quad \lambda\in{\mathbb C}, \Re\lambda>a.$$\end{thm}

Before proceeding further,  we would like to write down the following facts about (exponential) ultradistribution semigroups of $\ast$-class. If the operator $A$ generates an (exponential) ultradistribution semigroup $G$ of $\ast$-class, then there exists an injective operator $C\in L(E)$ such that the operator $A$ generates a (an exponentially bounded) $C$-regularized semigroup $(T(t))_{t\geq 0}$ on $E$ (cf. \cite[Section 3.6]{kniga}). This implies that there exist two Fr\' echet spaces $Y$ and $W$ such that $Y\hookrightarrow X\hookrightarrow W$ ($\hookrightarrow$ means the continuous embedding), as well as that $A_{|Y}$ generates a $C_{0}$-semigroup on $Y$ and that $A=B_{|X},$ where $B$ generates a $C_{0}$-semigroup on $W$; without going into full details, it should be only noticed that in $(M_{p})$-case $Y$ may be chosen to be exactly $E^{(M_{p})}(A),$ the abstract Beurling space of the operator $A$ (then $A_{|E^{(M_{p})}(A)}\in L(E^{(M_{p})}(A));$ cf. \cite[Theorem 3.6.6]{kniga} for more details). Furthermore, if the ultradistribution semigroup $G$ generated by $A$ is exponential, then $Y$ and $W$ may be chosen to be Banach spaces  (\cite{ralf}).

Let $\alpha>0,$ $\beta \in {\mathbb R},$ $\gamma \in (0,1),$ and let $m=\lceil \alpha \rceil.$ Then the Caputo fractional derivative
${\mathbf D}_{t}^{\alpha}u(t)$ is defined for those functions $u\in
C^{m-1}([0,\infty) : E)$ for which $g_{m-\alpha} \ast
(u-\sum_{k=0}^{m-1}u_{k}\frac{\cdot^{k}}{k!}) \in C^{m}([0,\infty) : E),$
by
$$
{\mathbf
D}_{t}^{\alpha}u(t)=\frac{d^{m}}{dt^{m}}\Biggl[g_{m-\alpha}
\ast \Biggl(u-\sum_{k=0}^{m-1}u_{k}\frac{\cdot^{k}}{k!}\Biggl)\Biggr].
$$
The
Mittag-Leffler function $E_{\alpha,\beta}(z)$ is defined by
$
E_{\alpha,\beta}(z):=\sum_{n=0}^{\infty}(z^{n}/\Gamma(\alpha
n+\beta)),$ $z\in {\mathbb C}.$
Here we assume that
$1/\Gamma(\alpha n+\beta)=0$ if $\alpha n+\beta \in -{{\mathbb
N}_{0}}.$ Set, for short, $E_{\alpha}(z):=E_{\alpha,1}(z),$ $z\in
{\mathbb C}.$
The Wright function
$\Phi_{\gamma}(\cdot)$ is defined by
$
\Phi_{\gamma}(t):={{\mathcal
L}^{-1}}(E_{\gamma}(-\lambda))(t),$ $t\geq 0,
$ where ${\mathcal L}^{-1}$ denotes the inverse Laplace transform.
We refer the reader to \cite{bajlekova} and \cite[Section 1.3]{knjigaho} for more details about Mittag-Leffler and Wright functions.

\section{Characterization of exponential ultradistribution semigroups}

The following proposition is motivated by D. Fujiwara's research  of exponential distribution semigroups (cf. \cite[Lemma 1]{fuji}).

\begin{prop}\label{lora}
\begin{itemize}
\item[(i)] Let $a_0\geq 0,$ and let $e^{-a_{0}\cdot}G\in L(\mathcal{S}^{(M_p)}(\mathbb{R}),L(E)).$ Then there exists $h'>0$ such that, for every $a>a_0$ and for every non-empty compact set $K\subseteq [0,\infty)$, there exists $C>0$ such that for any $\varphi,\psi\in\DD_0^{(M_{p})}$ with supp$(\varphi) \cup$ supp$(\psi)\subseteq K,$ the following estimate holds:
\begin{equation}\label{ocena}\bigl \|e^{-at}G(\varphi\ast\psi)\bigr \|_{L(E)}\leq C\|\varphi\|_{\DD_K^{\{M_p\},h'}}\|\psi\|_{L^1}.
\end{equation}
\item[(ii)] Let $a_0\geq 0,$ and let $e^{-a_{0}\cdot}G\in L(\mathcal{S}^{\{M_p\}}(\mathbb{R}),L(E)).$ Then, for every $h'>0,$ for every $a>a_0,$ and for every non-empty compact set $K\subseteq [0,\infty)$, there exists $C>0$ such that for any $\varphi,\psi\in\DD_K^{\{M_{p}\},h'},$ the estimate (\ref{ocena}) holds.
\end{itemize}
\end{prop}

\begin{proof}
Let $\zeta \in{\mathcal D}^{\{M_p\}}$, and let $\zeta(t)=1$ for all $t\in [-1, 1].$ Put $\gamma_1:=H\zeta$ ($H$ is the Heaviside function). Then
$\gamma'_1=\delta +H\zeta'$ in the sense of Roumieu ultradistributions. So, with $\omega_1=H\zeta'\in\DD_0^{\{M_p\}}$,
we have
\begin{equation}\label{paraal}
\delta'\ast\gamma_1=\delta+\omega_1.
\end{equation}
Let the numbers $h'>0$ and $a>a_{0}$ be fixed, let $K$ be a non-empty compact subset of $[0,\infty),$ and let $\varphi,\psi\in\DD_K^{\{M_{p}\},h'}.$ We will prove that the function $e^{(a_{0}-a)\cdot}(\varphi \ast \psi)(\cdot) \in {\mathcal S}^{M_{p}, h}({\mathbb R})$
for all $h\in (0,h'H/4(1+a-a_{0}));$ here $H$ denotes the constant from the condition $(M.2).$ Set $K':=(K+$supp$(\gamma_{1})) \cup (K+$supp$(\omega_{1}))$ and $K'':=K+K'.$ Using the product rule, Young's inequality and the condition $(M.2),$ we have that
\begin{align*}
& \bigl \|  e^{(a_{0}-a)\cdot}(\varphi \ast \psi)(\cdot) \bigr \|_{M_p,h}=\sup\Biggl\{\frac{h^{\alpha+\beta}}{M_\alpha M_\beta}\bigl(1+t^2\bigr)^{\beta/2}\Bigl| \bigl \langle \delta^{(\alpha)}\ast e^{(a_{0}-a)\cdot}(\varphi \ast \psi)(\cdot)\bigr \rangle(t)\Bigr|:t\in\mathbb{R},
\;\alpha,\;\beta\in\mathbb{N}_0\Biggr\}
\\ & =\sup \limits_{t\in\mathbb{R},
\;\alpha,\;\beta\in\mathbb{N}_0}\Biggl\{\frac{h^{\alpha+\beta}\bigl(1+t^2\bigr)^{\beta/2}}{M_\alpha M_\beta}\Bigl| \bigl \langle \delta^{(\alpha +1)} \ast \gamma_{1}\ast (e^{(a_{0}-a)\cdot} \varphi \ast e^{(a_{0}-a)\cdot} \psi)\bigr \rangle(t)- \bigl \langle \delta^{(\alpha)} \ast \omega_{1}\ast (e^{(a_{0}-a)\cdot} \varphi \ast e^{(a_{0}-a)\cdot} \psi)\bigr \rangle(t)\Bigr|\Biggr\}
\\ & \leq  \sup \limits_{t\in\mathbb{R},
\;\alpha,\;\beta\in\mathbb{N}_0}\Biggl\{\frac{h^{\alpha+\beta}\bigl(1+t^2\bigr)^{\beta/2}}{M_\alpha M_\beta}
 \int_{{\mathbb R}}\Bigl|\Bigl(e^{(a_{0}-a)\cdot} \varphi(\cdot)\Bigr)^{(\alpha +1)}(t-s)\Bigr| \ \Bigl| \bigl( e^{(a_{0}-a)\cdot}\psi \ast \gamma_{1}\Bigr)(s)\Bigr|\, ds
\Biggr\}
\\ & + \sup \limits_{t\in\mathbb{R},
\;\alpha,\;\beta\in\mathbb{N}_0}\Biggl\{\frac{h^{\alpha+\beta}\bigl(1+t^2\bigr)^{\beta/2}}{M_\alpha M_\beta} \int_{{\mathbb R}}\Bigl|\Bigl(e^{(a_{0}-a)\cdot} \varphi(\cdot)\Bigr)^{(\alpha)}(t-s)\Bigr| \ \Bigl| \bigl( e^{(a_{0}-a)\cdot}\psi \ast \omega_{1}\Bigr)(s)\Bigr|\, ds \Biggr\}
\\ & =\sup \limits_{t\in K'',
\;\alpha,\;\beta\in\mathbb{N}_0}\Biggl\{\frac{h^{\alpha+\beta}\bigl(1+t^2\bigr)^{\beta/2}}{M_\alpha M_\beta}
 \int_{{K'}}\Bigl|\Bigl(e^{(a_{0}-a)\cdot} \varphi(\cdot)\Bigr)^{(\alpha +1)}(t-s)\Bigr| \ \Bigl| \bigl( e^{(a_{0}-a)\cdot}\psi \ast \gamma_{1}\Bigr)(s)\Bigr|\, ds
\Biggr\}
\\ & + \sup \limits_{t\in K'',
\;\alpha,\;\beta\in\mathbb{N}_0}\Biggl\{\frac{h^{\alpha+\beta}\bigl(1+t^2\bigr)^{\beta/2}}{M_\alpha M_\beta} \int_{K'}\Bigl|\Bigl(e^{(a_{0}-a)\cdot} \varphi(\cdot)\Bigr)^{(\alpha)}(t-s)\Bigr| \ \Bigl| \bigl( e^{(a_{0}-a)\cdot}\psi \ast \omega_{1}\Bigr)(s)\Bigr|\, ds \Biggr\}
\\ & \leq e^{M\bigl(h\sqrt{1+(\sup K'')^{2}}\bigr)}\sup \limits_{t\in K'',
\;\alpha\in\mathbb{N}_0}\Biggl\{\frac{h^{\alpha}}{M_\alpha}
 \int_{{K'}}\Bigl|\Bigl(e^{(a_{0}-a)\cdot} \varphi(\cdot)\Bigr)^{(\alpha +1)}(t-s)\Bigr| \ \Bigl| \bigl( e^{(a_{0}-a)\cdot}\psi \ast \gamma_{1}\Bigr)(s)\Bigr|\, ds
\Biggr\}
\end{align*}
\begin{align*}
\\ & + e^{M\bigl(h\sqrt{1+(\sup K'')^{2}}\bigr)} \sup \limits_{t\in K'',
\;\alpha \in\mathbb{N}_0}\Biggl\{\frac{h^{\alpha}}{M_\alpha} \int_{K'}\Bigl|\Bigl(e^{(a_{0}-a)\cdot} \varphi(\cdot)\Bigr)^{(\alpha)}(t-s)\Bigr| \ \Bigl| \bigl( e^{(a_{0}-a)\cdot}\psi \ast \omega_{1}\Bigr)(s)\Bigr|\, ds \Biggr\}
\\ & \leq 2(1+a-a_{0})e^{M\bigl(h\sqrt{1+(\sup K'')^{2}}\bigr)}\sup \limits_{
\;\alpha\in\mathbb{N}_0}\Biggl\{\frac{(2(1+a-a_{0})h)^{\alpha}}{M_\alpha}
\sum_{j=0}^{\alpha+1}\bigl \|\varphi^{(j)}\|_{L^{\infty}(K)}  \ \bigl\| e^{(a_{0}-a)\cdot}\psi \ast \gamma_{1}\bigr\|_{L^{1}({\mathbb R})}
\Biggr\}
\\ & + 2(1+a-a_{0})e^{M\bigl(h\sqrt{1+(\sup K'')^{2}}\bigr)}\sup \limits_{
\;\alpha\in\mathbb{N}_0}\Biggl\{\frac{(2(1+a-a_{0})h)^{\alpha}}{M_\alpha}
\sum_{j=0}^{\alpha}\bigl \|\varphi^{(j)}\|_{L^{\infty}(K)}  \ \bigl\| e^{(a_{0}-a)\cdot}\psi \ast \omega_{1}\bigr\|_{L^{1}({\mathbb R})}
\Biggr\}
\\ & \leq \frac{4}{A}(1+a-a_{0})e^{M\bigl(h\sqrt{1+(\sup K'')^{2}}\bigr)}\bigl\| \psi \bigr \|_{L^{1}({\mathbb R})} \Bigl[\bigl \| \gamma_{1}\bigr\|_{L^{1}({\mathbb R})}+\bigl \| \omega_{1}\bigr\|_{L^{1}({\mathbb R})}\Bigr]
\\ & \times \|\varphi\|_{\DD_K^{\{M_p\},h'}} \sup \limits_{
\;\alpha\in\mathbb{N}_0}\Biggl\{
\sum_{j=0}^{\alpha+1} \Bigl(\frac{2(1+a-a_{0})h}{h'}\Bigr)^{j}\frac{(2(1+a-a_{0})h)^{\alpha -j}M_{j}}{H^{j}M_{j}M_{\alpha-j}}
\Biggr\}
\\ & \leq \frac{4}{A}(1+a-a_{0})e^{M\bigl(h\sqrt{1+(\sup K'')^{2}}\bigr)}\bigl\| \psi \bigr \|_{L^{1}({\mathbb R})} \Bigl[\bigl \| \gamma_{1}\bigr\|_{L^{1}({\mathbb R})}+\bigl \| \omega_{1}\bigr\|_{L^{1}({\mathbb R})}\Bigr]
\\ & \times \|\varphi\|_{\DD_K^{\{M_p\},h'}} e^{M(2(1+a-a_{0})h)}
\sum_{j=0}^{\infty} \Bigl(\frac{2(1+a-a_{0})h}{h'H}\Bigr)^{j}
\\ & \leq \frac{4}{A}(1+a-a_{0})e^{M\bigl(h\sqrt{1+(\sup K'')^{2}}\bigr)}\bigl\| \psi \bigr \|_{L^{1}({\mathbb R})} \Bigl[\bigl \| \gamma_{1}\bigr\|_{L^{1}({\mathbb R})}+\bigl \| \omega_{1}\bigr\|_{L^{1}({\mathbb R})}\Bigr] \|\varphi\|_{\DD_K^{\{M_p\},h'}} e^{M(h'H/2)}.
\end{align*}
The final conclusion follows from the decomposition $\langle e^{-a\cdot }G , \varphi\ast \psi \rangle=\langle e^{-a_{0} \cdot }G, e^{(a_{0}-a)\cdot}(\varphi\ast \psi)\rangle$ and the continuity of mapping $e^{-a_{0} \cdot }G : {\mathcal S}^{M_{p},h}\rightarrow L(E),$ with the number $h\in (0,h'H/4(1+a-a_{0}))$ chosen arbitrarily.
\end{proof}

\begin{rem}
\begin{itemize}
\item[(i)]
Notice that we have not used the structural theorem for the space of tempered vector-valued ultradistributions in the proof (cf.  \cite[Theorem 2.2]{kpsv}) and that we do not need the condition $(M.3)$ here. Observe also that the proof of Proposition \ref{lora} works in the case that $E$ is a general sequentially complete locally convex space (\cite{novi}).
\item[(ii)] As mentioned above, Proposition \ref{lora} is inspired by Lemma 1 of \cite{fuji}. Regrettably, it seems that the assertions of \cite[Theorem 2-Theorem 3]{fuji} cannot be so simply reconsidered for exponential ultradistribution semigroups. Speaking-matter-of-factly, in ultradistribution case we cannot prove the existence of number $N$ such that an analog of the estimate \cite[(9)]{fuji} holds.
\end{itemize}
\end{rem}

In the remaining part of this section, we shall always assume that the condition $(M.3)$ holds.
The following entire function of exponential type zero (\cite{ci1}, \cite{k91}) plays an important role in our further work:
$$
\omega^{(M_{p})}(z)=:\prod_{i=1}^{\infty}\bigl(1+\frac{iz}{m_p}\bigr),\quad z\in\mathbb{C}.
$$
Then $| \omega^{(M_{p})}(z)  | \geq e^{M(|z|)},$ $z\in {\mathbb C}$
and
the properties (P.1)-(P.5) stated in \cite[Subsection 3.6.2]{kniga} hold.

Suppose now that $A$ generates an exponential ultradistribution semigroup of $(M_{p})$-class. By Theorem \ref{holh}, it readily follows that there exist numbers $a\geq0$, $k>0$ and $L>0$ such that $\{\lambda\in{\mathbb C} : \Re\lambda>a\}\subseteq\rho(A)$ and
$\|R(\lambda:A)\|\leq Le^{M(k|\lambda|)},$ $\lambda\in{\mathbb C},\ \Re\lambda>a .$ Let $\bar{a}>a.$ Then there exists a sufficiently large number $n_{0} \in {\mathbb N}$ such that, for every $n\in {\mathbb N}$ with $n\geq n_{0},$ the operator $A$ generates a global exponentially bounded $C_{n}$-regularized semigroup $(S_{n}(t))_{t\geq 0}$ on $E,$ where
$$
S_{n}(t)x:=\frac{1}{2\pi i}\int \limits_{\bar{a}-i\infty}^{\bar{a}+i\infty}e^{\lambda t}\frac{R(\lambda : A)x}{\omega^{n}(i\lambda)}\, d\lambda,\quad t\geq 0,\ n\in {\mathbb N},\ x\in E\mbox{ and }C_{n}:=S_{n}(0).
$$
cf. the proof of \cite[Theorem 3.6.4]{kniga} (in this result, we have clarified the result on generation of {\it local} $C$-regularized semigroups; observe, however, that in the case of exponential ultradistribution semigroups the resulting $C_{n}$-regularized semigroup is global and exponentially bounded because the integration is taken along the line connecting the points $\bar{a}-i\infty$ and $\bar{a}+i\infty,$ not along the upwards oriented boundary of a suitable ultra-logarithmic region of $(M_{p})$-class). An elementary application of Cauchy's formula, taken together with the consideration in the paragraph containing the equation \cite[(316)]{kniga} and \cite[Lemma 3.6.5, Theorem 3.6.6]{kniga}, implies that
$E^{(M_{p})}(A)\subseteq C_{n}(D_{\infty}(A)),$ $n\geq n_{0}$ and that for each $x\in E^{(M_{p})}(A)$ a unique solution of the abstract Cauchy problem
$$
(\text{\emph{ACP}}):\left\{
\begin{array}{l}
u\in C^{\infty}([0,\infty):E)\cap C([0,\infty):[D(A)]),\\
u'(t)=Au(t),\;t\geq 0,\\ u(0)=x,
\end{array}
\right.
$$
is given by $u(t)=S_{n}(t)C_{n}^{-1}x,$ $t\geq 0,$ $n\geq n_{0}.$ Using the same argumentation as in the proof of \cite[Theorem 3.6.4]{kniga}, we get that for each $x\in E^{(M_{p})}(A)$ and $h>0$ there exist $n\geq n_{0}$ and $c>0$ such that
\begin{equation}\label{jednazba}
\sup_{p\in {{\mathbb N}_{0}}}\frac{h^{p}\bigl \| \frac{d^{p}}{dt^{p}}u(t) \bigr \|}{M_{p}}\leq ce^{\bar{a}t}\bigl\| C_{n}^{-1}x\bigr \|,\quad t\geq 0.
\end{equation}
Applying \cite[Theorem 2.4.2]{knjigaho} (cf. also \cite[Theorem 3.1, Theorem 3.3]{bajlekova}) and the Ljubich uniqueness theorem for abstract time-fractional equations \cite[Theorem 2.1.34]{knjigaho}, it readily follows that
for each $x\in E^{(M_{p})}(A)$ and $\alpha \in (0,1)$ there exists a unique solution of the following abstract fractional Cauchy problem
$$
(\text{\emph{ACP}})_{\alpha}:\left\{
\begin{array}{l}
v\in C([0,\infty):E)\cap {\mathcal A}_{\min((\frac{1}{\alpha}-1)\frac{\pi}{2},\pi)}(E),\\
{\mathbf D}_{t}^{\alpha}v(t)=Av(t),\;t\geq 0,\\ v(0)=x,
\end{array}
\right.
$$
given by $v(t):=\int_{0}^{\infty}\Phi_{\alpha}(s)S_{n}(st^{\alpha})C_{n}^{-1}x\, ds,$ $t\geq 0$ ($n\geq n_{0}$). Thus, we have proved the following theorem.

\begin{thm}\label{glass}
Suppose that $(M.3)$ holds for $(M_{p})$ and $A$ generates an exponential ultradistribution semigroup of $(M_{p})$-class. Then for each
$x\in E^{(M_{p})}(A)$ there exists a unique solution $u(t)$ of the abstract Cauchy problem
$
(ACP)$ satisfying additionally that for each $h>0$ there exist $n\geq n_{0}$ and $c>0$ such that (\ref{jednazba}) holds. Furthermore, for each $x\in E^{(M_{p})}(A)$ and $\alpha \in (0,1)$ there exists a unique solution of the abstract fractional Cauchy problem $
(ACP)_{\alpha}.$
\end{thm}

It is worth noting that the assertion of Theorem \ref{glass} continues to hold, with suitable modifications, in the setting of Fr\' echet spaces and that possible applications can be made to differential operators considered in \cite[Example 3.5.15]{kniga} and \cite[Example 9.3]{novi}.

\end{document}